\documentclass[12pt,titlepage, a4paper, draft]{article}

\usepackage{amssymb,latexsym,amsfonts,amsmath,amsthm}

\usepackage{verbatim}
\usepackage{url}

{\makeatletter \@addtoreset{equation}{section}}

\newcommand{\diag}{\mathop{\mathrm{diag}}\nolimits}
\newcommand{\Ker}{\mathop{\mathrm{Ker}}\nolimits}
\newcommand{\IIm}{\mathop{\mathrm{Im}}\nolimits}

\newcommand{\Inv}[1][\,]{\mathop{\mathrm{Inv}#1}\nolimits}

\newcommand{\mcU}{{\mathcal{U}}}

\newcommand{\h}{ \mathop{ \mathrm{h} {} }\nolimits }
\newcommand{\e}{ \mathop{ \mathrm{e} {} }\nolimits }

\newcommand{\Aut}{ \mathop{ \mathrm{Aut} {} }\nolimits }

\newcommand{\beq}{\begin{equation}}
\newcommand{\eeq}{\end{equation}}

\newcommand{\LR}{\Leftrightarrow}
\newcommand{\Ra}{\Rightarrow}

\newcommand{\bpm}{\begin{pmatrix}}
\newcommand{\epm}{\end{pmatrix}}

\newtheorem{theorem}{Theorem}[section]
\newtheorem{lemma}[theorem]{Lemma}

\newtheorem{example}[theorem]{Example}

\usepackage[dvips]{color}

\definecolor{darkgreen}{rgb}{0,0.7,0}
 \definecolor{exp}{rgb}{0.85,0.7,0}
  \definecolor{msk}{rgb}{0.65,0.3,0.3}

\definecolor{msk2}{rgb}{0.75,0.3,0.1}
\definecolor{r}{rgb}{0.75,0.1,0.4}

\definecolor{exp2}{rgb}{0.85,0.6,0.1}

\definecolor{vvs}{rgb}{0.5,0.4,0.7}

\begin{document}

\title{Characteristic and hyperinvariant subspaces over
the field GF(2)}

\author{Pudji Astuti
\\
Faculty of Mathematics\\
 and Natural Sciences\\
Institut Teknologi Bandung\\
Bandung 40132\\
Indonesia
         \and
Harald~K. Wimmer\\
Mathematisches Institut\\
Universit\"at W\"urzburg\\
97074 W\"urzburg\\
Germany}

\date{\today}

\maketitle

\begin{abstract}

\vspace{2cm}

\vspace{.5cm}



\vspace{1cm}

\noindent
{\bf Abstract:}
Let $f$ be an endomorphism of  a vector space  $V$
over a field $K$.
An $f$-invariant subspace $X \subseteq V$ is called
 hyperinvariant (respectively characteristic)   if
$X$ is invariant under all endomorphisms
(respectively  auto\-morphisms) that commute with $f$.
If $|K| > 2$ then all characteristic subspaces
are hyperinvariant.
If  $|K| = 2$ then there are endomorphisms $f$
with invariant subspaces that are characteristic
but not hyperinvariant.
In this paper we give a
new proof of a theorem of Shoda,  which provides
a necessary and sufficient condition for the
existence of characteristic  non-hyperinvariant
subspaces.

\vspace{3cm}
\noindent
{\bf Address for Correspondence:}\\
H. Wimmer\\
Mathematisches Institut \\
Universit\"at W\"urzburg\\
Am Hubland\\
97074 W\"urzburg\\
Germany

\vspace{.2cm}
\flushleft{
{\textsf e-mail:}~~\texttt{\small wimmer@mathematik.uni-wuerzburg.de} }\\
{\textsf Fax:}~~
+49 931 8\,88\,46\,11\\

\end{abstract}

\section{Introduction} \label{sct.int}

Let $V$ be an $n$-dimensional vector space over a field $K$ and
let  $ f : V \to V $ be  $K$-linear.
The set of  $f$-invariant subspaces of $V$ form a lattice,
which we denote by   $ {\rm{Inv}}( V) $.
In this paper we are concerned with two sublattices of $ {\rm{Inv}}( V) $.
If a  subspace $X $  of $V$     remains
invariant for all endomorphisms of $V$ that commute
with $f$ then  $X$ is  called
{\em{hyperinvariant}} for $f$ \cite[p.\ 305]{GLR}.
We say  \cite{AW4}
that a subspace  $X$ of $V$  is
{\em{characteristic}} for $f$   if
$X \in  {\rm{Inv}}( V) $ and
\,$ \alpha( X ) = X $\, for all $K$-automorphisms
$\alpha$ of $V$ that commute with $f$.
Let \,$   {\rm{Hinv (V) }}$\, and
$  {\rm{Chinv (V) }} $ be  set of hyperinvariant and
 characteristic  subspaces of  $V$, respectively.
Both sets are lattices, and  \,$   {\rm{Hinv (V) }}
 \subseteq   {\rm{Chinv (V) }}  $.
If the      the characteristic
  polynomial of $f $ splits over $K$ (such that all eigenvalues of
$f$ are in $ K$)
 then one can reduce the study
of  \,$   {\rm{Hinv (V) }} $
 and   \,${\rm{Chinv (V) }}  $ to the case where
$ f $ has only one eigenvalue, in particular to the
case where $f$ is nilpotent.
Thus, throughout  this paper we shall assume \,$ f^n = 0$.
Then  (see for example \cite[p.\,306]{GLR})
the lattice
 \,$   {\rm{Hinv (V) }} $\, is the smallest
sublattice of $ {\rm{Inv}}( V) $ that contains
\beq \label{eq.imk}
  \Ker f ^k,  \,\, \IIm f^k ,  \: k = 0, 1, \dots , n.
\eeq
It is well  known  (\cite{Kap}, \cite{Sh},  \cite{AW4})
that each characteristic subspace
is hyper\-invariant if  $|K| > 2$.
In this paper we consider vector spaces $V$ over the
field $K = GF(2) $ and we focus on   characteristic subspaces that
are not hyperinvariant.
The following example shows that in the case of $K = GF(2) $
it may occur that \,$  {\rm{Chinv (V) }} \supsetneq   {\rm{Hinv (V) }}$.
 The set of  endomorphisms, respectively  automorphisms,  of $V$
that commute with $f$ will be denoted by
  $ {\rm{End}}_f (V) $, respectively $ {\rm{Aut}}_f (V) $.

\begin{example} \label{ex.fu22}
{\rm{Let $ K =  GF(2) = \{0,1\} $.  Consider  $ V = K ^4 $
 and
\[
f= \left(
\begin{array}{c | ccc}   0 &  0 & 0   & 0
\\ \hline
 0 &  0 & 0   & 0 \\
 0 &  1 & 0  & 0 \\
              0 &    0 & 1 & 0
\end{array} \right)  .
\]
 Let   $e_1, \dots , e_4$,
be the unit vectors of $K^4$.
 Set \,$ z =  e_1 + e_3  = (1, 0, 1, 0 )^T $.
 Then $ f^2 z = 0 $. Define
 \,$X = {\rm{span}}\{z, fz\}$.
Then
\beq \label{eq.ouszb}
    X = \{0, z  , fz  , z + fz \}  = \{ 0, e_1 + e_3,  e_4 ,
 e_1 + e_3+  e_4 \} \,  \in \,  \Inv(V) .
\eeq
Note that
$  {\rm{End}}_f (V) $ consists of all matrices of the form
\beq \label{eq.cmm}
g =
\bpm  a & b & 0 & 0
\\
      0 & c & 0 & 0
\\
      0 & d & c & 0
\\
      k & h & d & c
\epm ,
\eeq
and  \,$ g \in    {\rm{Aut}}_f (V) $\,  if and only if

\beq \label{eq.alp}
g=
\bpm  1 & b & 0 & 0
\\
      0 & 1 & 0 & 0
\\
      0 & d & 1 & 0
\\
      k & h & d & 1
\epm .
\eeq
Thus, if  $ g \in    {\rm{Aut}}_f (V) $ then
\[
 gz
 =  e_1 + e_3 + (k+ d) e_4 , \;
g fz
= e_4 ,  \;
g ( z + fz)
 = e_1 + e_3 +  (k+ d + 1) e_4  .
\]
Hence $ gX \subseteq X $,  and
therefore $X \in  {\rm{Chinv}} (V)  $.
Let   $ \pi _1  = \diag (1, 0, 0, 0 ) $ be the orthogonal
projection on $ K e_1 $. Then $ \pi _1   \in
{\rm{End}}_f (V) $, and  we have  $ \pi _1  z = e_1 $,
but $ e_1 \notin X $. Therefore $X$ is not hyperinvariant.
}}
\end{example}

The example is in accordance with
the following result of Shoda \cite[Satz~5, p.\,619]{Sh},
 \cite{Kap0},   \cite[p.\,63/64]{Kap},  which we state in terms
of the Jordan normal form of $f$.

\begin{theorem}    \label{thm.vnpsa}     
Let $ V $ be a finite-dimensional vector space over
  the field  \,$ K =  GF(2) $ and let $ f : V \to V $ be nilpotent.
The following statements are equivalent.
\\
{\rm{(i)}}
There exists a characteristic subspace of $V$
 which is not hyperinvariant.
\\
{\rm{(ii)}}
For some  numbers $ r  $ and $ s $ with
   \,$s    >    r + 1  $\,
the Jordan form of $f$ contains  exactly one Jordan block of size
$s $ and exactly one block of size $r$.
\end{theorem}

It is the main purpose of this paper
to give a new proof  of  the implication
``(i) $\Ra$ (ii)'' of
Shoda's theorem.
In order to present a complete picture
we also include a proof of the reverse implication
``(ii) $\Ra$ (i)''.
 Theorem~\ref{thm.chhym} below
plays a key role in our new proof .
It  relates characteristic and hyperinvariant subspaces
with  marked subspaces.
Recall that an $f$-invariant subspace $W \subseteq V $, $ W \ne 0 $,
is said to be  {\em{marked}}  \cite[p.\,83]{GLR} if it has a
Jordan basis (with respect to $f_{|W}$) that can be extended
to a Jordan basis of $V$. The zero subspace 
 is assumed to be marked.

\begin{theorem} {\rm{\cite[p.\,268]{AW4}}} \label{thm.chhym}
Let $W \in \Inv(V)$. Then $ W $ is a hyperinvariant subspace if and only if
$W$ is characteristic and marked.
\end{theorem}

\medskip

The proof of  Theorem~\ref{thm.vnpsa} will be divided into several parts.
It relies on results on characteristic subspaces in
Section~\ref{sct.chrst}.
Auxiliary material
and basic facts  on generator tuples and marked
subspaces  are discussed  in Section~\ref{sct.axlr}.

\section{Definitions and notation} \label{sct.dfn}

We set $  V[ f^j ] =  \Ker f^j $, $j \ge 0 $.
Clearly,  $ f^n = 0 $ implies $ V =  V[ f^n]$.
Define $ \iota = {\rm{id}}_V $ and $ f^0 =  \iota$.
Let $ x \in V $. The smallest nonnegative integer $\ell$
 with
$f^{\ell} x = 0$
is  called
the {\em{exponent}} of $x$. We write  $\e(x) = \ell$.
A  nonzero vector  $x $
  is said to have \emph{height} $q $ if $x \in f^q V$
and $x \notin f^{q+1} V$.
In this case we write $\h(x) = q$. We set $ \h ( 0 ) =  \infty $.
Let $ Y \subseteq V$. We write $ \e(Y) = s $
 if $ \e(y) = s $ for all $y \in Y $.
Moreover,  $ \h(Y)  = q $ shall mean $ Y \subseteq f^qV $ together
with  $ Y \nsubseteq f^{q+1}V $.
Let
\[
 \langle x  \rangle \,  = \,  {\rm{span}} \{ f^i x , \, i \ge 0 \}
\]
 be  the  cyclic subspace generated by $ x $.
If \,$ B  \subseteq V$\, we define
 \,$  \langle B \rangle  =
 \sum _{b \, \in \,  B } \,  \langle \,  b  \, \rangle  $.
We   call  \,$ U = ( u_1, \dots , u_k) $\,   a
{\em{generator tuple}} of $V$  if
\[  V =
 \langle u_1 \rangle
\,  \oplus \,
\cdots \,
 \oplus    \langle u_{ k } \rangle
\]
 and if  the elements of $U$ are ordered
 such that
\beq \label{eq.rodr}
 t_1 \,  = \,  \e( u_1 ) \,\,   \le \,\,  \cdots \,\,  \le \,\,
\e( u_k )  \,  = \, t_k  .
\eeq
Let  $ \mathcal{U} $ be the set of   generator tuples  of $V$.
Thus
\beq   \label{eq.elt}
 s^{ t_{1}}, \,
\dots , \,  s ^{t_{k} } ,  \, \,\,
 0 < t_{1}  \, \le  \, \cdots  \,   \le  \,  t_{k},
\eeq
are the elementary divisors of $f$, and
\,$ t_1 + \cdots + t_k = \dim V = n $\, and \,$ \dim \Ker f = k $.
Set
\beq \label{eq.npjo}
N_{t}  = \bpm  0 &  &  &  &  &   &
\\
1 & 0  &  &  &  &                &
\\
 & 1 &  0 &  &  &                 &
\\
 &  &   \ddots   & \ddots  &   &  &
\\
 &  &  &  &     &  0                   &
\\
&   &  &  &  & 1 & 0
\\
\epm _{t \times t} .
\eeq
Then  \,$J = \diag(N_{t_1}, \dots , N_{t_k} ) $\, is the
Jordan form of $f$.
We define
\[
d(r) =
\dim \big(  V[f] \cap
  f ^{r -1} V \,\,  /  \, \, V[f]  \cap  f ^{r}  V \big),
\:\:
 r = 1, 2,  \dots , n.
\]
 Using the terminology of
  abelian $p$-groups \cite{FuI}  or $p$-modules
\cite[p.27]{Kap}
we call $ d(r) $ the $r-1$-th   {\em{Ulm invariant}}
and
\beq \label{eq.ulmsq}
D = \big(  d(1) , \dots , d(n) \big)
\eeq
the  {\em{Ulm sequence}} of the pair $ (V, f) $.
Then $   d (r) $ is equal to  the number of Jordan blocks
of size $r  $ in the Jordan form $J$ of $f$.
  If $V  = \langle x \rangle $, $ \e(x) = n $,
then $  D = (0, \dots, 0 , 1)$. Thus,
 if $ V[f] = f^{a-1} V $ then
$  D = (0, \dots, 0 , n/a, 0,  \dots , 0)$.
\medskip

Let  \,$ U = ( u_1, \dots , u_k) \in \mcU $.
It will be convenient to partition  $ U $ into
subsets of equal exponent.
Denote the distinct elements of
\,$\{\e(u_r); \, 1 \,\le \, r \, \le  \, k\}$\,
 by \,$a_1, a_2, \dots, a_m$\,
labelled so that \,$a_1 < \cdots < a_m$\,
and  set
\[
 U_{ a_{\mu } } = \{ u \in U ; \:  \e ( u )  =  a_{\mu} \} ,
\; \mu = 1, \dots , m .
\]
Then
\beq \label{eq.prtt}
 U  =
\bigl( U_{a_1}, \dots , U_{a_m} \bigr)
\quad {\rm{and}} \quad
\e( U_{ a_{1} } ) =  a_{1}\, < \,\cdots \,
 < \, \e( U_{ a_{m} } )  = a_m ,
\eeq
and \,$ | U_{ a_{\mu} } | = d( a_{\mu}) $.
To \eqref{eq.prtt} corresponds the
decomposition
\beq \label{eq.fvp}
 V \,  = \,  \langle U_{a  _{1 }}  \rangle \,
\oplus \,  \cdots \, \oplus \,
  \langle U_{a _{m }}  \rangle.
\eeq
Let
\,$ \pi _{\mu} : V \to V $\, be the projection
with
\[
\pi _{\mu}  V =  \langle U_{a_{\mu}}    \rangle ,
\;\:
\Ker \pi _{\mu} = \langle U_{a_1}, \dots ,
 U_{a_{\mu -1}} ,  U_{a_{\mu +1}}, \dots , U_{a_m}  \rangle .
\]
Note that $  \pi _{\mu} \in  {\rm{End}}_f(V) $.

\section{Auxiliary results} \label{sct.axlr}

\subsection{Automorphisms and generators}

The following lemma shows that
   each $\alpha \in  {\rm{Aut}}_f(V) $ is uniquely determined
by the image of a given  generator tuple.

\begin{lemma} {\rm{\cite{AW4}}}  \label{la.ausoam}
Let    \,$ U = ( u_1, \dots , u_k)
 \in \mathcal{U} $  be given.
For \,$    \alpha \in  {\rm{Aut}}_f(V) $\, define
\,$
\Theta_U( \alpha)    = \bigl( \alpha ( u_1 ) ,
  \dots ,  \alpha ( u_k  ) \bigr)$.
Then
\[  \alpha \mapsto  \Theta  _U  ( \alpha) , \,\,
 \Theta _U :  {\rm{Aut}}_f(V) \to  \mathcal{U} ,
\]
is a bijection.
\end{lemma}

\medskip


The next  lemma will enable us to
exchange  vectors
 in a  generator tuple.

\begin{lemma}  \label{la.pd13}
Suppose
$ V = \langle u_1 \rangle  \oplus
\cdots  \oplus
 \langle u_n  \rangle $ and
$ \e(u_i) = a $, $ i = 1, \dots , n$.
If $ x \in V$, $ x \ne 0 $, and $ \h(x) = 0 $, then there exist an
index $j$ such that
\beq \label{eq.uu}
 (u_1, \dots , u_{j-1}, x ,  u_{j+1}, \dots , u_n )  \in \mcU.
\eeq
\end{lemma}

\begin{proof}
Let
$ x = x_1 + \cdots + x_k $, $ x_i \in  \langle u_i \rangle $.
Then
\[
 x_i = \sum _{\nu = 0} ^{a-1 } c_{  i \nu  } f ^{\nu } u_i
=
   ( u_i, f u_i \dots , f^{a-1 } u_i ) \begin{pmatrix}
 c_{i 0 } \\  c_{i 1 } \\ \vdots \\ c_{i, a -1 }
\end{pmatrix} .
\]
Set
\[
   C_i =
\begin{pmatrix}
  c_{i 0 }         &               &                  &        &
\\
 c_{i 1 }          &   c_{i 0 }    &                 &         &
\\
 c_{i 2 }          &   c_{i 1 }    &  c_{i 0 }       &       &
\\
.                  &     .         &        .        &  .     &
\\
c_{i, a -1 }       &   c_{i, a -2 }&   c_{i, a -3 }  & \dots &  c_{i 0 }
 \end{pmatrix} .
\]
Then
\[
 ( x_i, f x_i \dots , f^{a-1 } x_i ) =
 ( u_i, f u_i \dots , f^{a-1 } u_i ) C_i ,
\]
and
\[
   ( x, f x \dots , f^{a-1 } x ) =
\sum _{i = 1} ^n
         ( u_i, f u_i \dots , f^{a-1 } u_i ) C_i .
\]
Because of
$ \h(x ) = 0 $ we have   $ \h(x_j) = 0 $
 for some $j$.  Thus  $ c_{j0} \ne 0 $,
and
$ C_j $ is nonsingular.
We obtain
\begin{multline} \label{eq.ubf}
      ( u_j, f u_j \dots , f^{a-1 } u_j )
=
\\
  ( x, f x \dots , f^{a-1 } x ) C_j ^{-1}
-
\sum _{i \ne j}   ( u_i, f u_i \dots , f^{a-1 } u_i ) C_i  C_j ^{-1}.
\end{multline}
The vectors
\beq \label{eq.jb}
  B
 =
(  u_1, f u_1 \dots , f^{a-1 } u_1 , \dots ,
 u_k, f u_k \dots , f^{a-1 } u_n )
\eeq
are a Jordan basis of $V$.
Because of \eqref{eq.ubf} we obtain another Jordan basis if
we replace the vectors
  $ ( u_j, f u_j \dots , f^{a-1 } u_j ) $
in $B$   by
$   ( x, f x \dots , f^{a-1 } x ) $.
This proves \eqref{eq.uu}.
\end{proof}

\medskip

In  the proof of Theorem~\ref{thm.lgsth}
we shall use the following observation.

\begin{lemma} \label{la.mma8}
Let    \,$U= (U_{a_1}, \dots, U_{a_m}) \in \mcU$.
Suppose  \,$i <j$\, and let
 $w \in  U_{a_i}$,  $y \in  U_{a_j} $.
Then there exists $\alpha \in  {\rm{Aut}}_f(V) $ such
that
$ \alpha y = w + y$.
\end{lemma}

\begin{proof}
From  \,$ y, w \in U $\,
follows \,$ \h(y + w) = 0$, and
\,$ \e( U_{a_i} ) < \e (  U_{a_j} ) $\,
implies \,$\h(f^{a_j -1} ( y+ w)) = a_j -1$ and
$ \e ( y+ w) = a_j $.
We can assume $  U_{a_j} = (y, y_2, \dots , y_r) $.
Set $ \tilde{U}_{a_j} = (y + w, y_2, \dots , y_r) $.
If we
replace $  U_{a_j} $ in $\mcU $ by $ \tilde{U}_{a_j} $
we obtain another generator tuple $\tilde{U} \in \mcU$.
Then Lemma~\ref{la.ausoam}  yields the desired automorphism.
\end{proof}

\subsection{Marked subspaces}

Marked subspaces can be traced back to \cite[p.\,83]{GLR}.
They
have  been studied in  \cite{Bru}, \cite{FPP}, \cite{AW3},
and         \cite{CFP}.
Let    $ s^{t_i} $, $0 < t_1 \le \cdots \le t_k$,
be the elementary divisors of $f$.
We say that a $k$-tuple
 $r = (r_1, \dots , r_k) $  of
integers  is {\em{admissible}}   if
\beq \label{eq.adr}
 0  \, \leq \, r_i \, \leq \, t_i, \,\,\, i = 1, \dots , k.
\eeq
Each   $U \in  \mathcal{U}$   together with an
 admissible tuple $r$ gives rise
to a subspace
\beq \label{eq.wru}
W(r, U)
=
 \langle f ^{r_1}  u_1 \rangle  \,  \oplus  \,  \cdots  \,   \oplus
 \,  \langle f ^{r_k}  u_k \rangle  ,
\eeq
which is marked  in $V$. Conversely,
 a subspace $W $ is marked in $V$
only if $ W = W(r, U) $ for some  $U \in  \mathcal{U}$
   and some admissible  $ r$.
The next theorem describes those
subspaces  $ W(r, U) $ that are independent of the
generator tuple $U$.

\bigskip
\begin{theorem} {\rm{(\cite{AW4}, \cite[p.\ 162]{Lo})}} \label{thm.qv}
Let \,$ U \in \mathcal{U} $\,
be given as in  \eqref{eq.rodr},
  and let $r  = (r_1, \dots , r_k) $\,
be admissible. Then the following statements are equivalent.
\begin{itemize}
\item[\rm{(i)}] The subspace  $  W(r, U) $ is characteristic.
\item[\rm{(ii)}]
The tuples $t = ( t_1, \dots , t_k)$ and  \,$r = (r_1, \dots , r_k) $\,
satisfy
  \beq  \label{eq.r12}
 r_1  \leq  \cdots  \leq  r_k \quad and \quad
  t_1 - r_1   \leq   \cdots   \leq  t_k - r_k .
\eeq
\item[\rm{(iii)}] The subspace  $  W(r, U) $ is hyperinvariant.
\end{itemize}
\end{theorem}

Note that  \eqref{eq.r12}
implies that  $r_i = r_j $ if $ t_i = t_j$.
Let $ U = (U_{a_1}, \dots , U_{a_m} )$, $ a_1 < a_2 < \cdots < a_m$.
Hence if $ X $ is characteristic then $X $ is marked
if and only if
\beq \label{eq.rpcs}
 X =  f^{c _1}   \langle  U_{a_1} \rangle
\oplus \cdots \oplus
 f^{c_m}    \langle U_{a_m}  \rangle
\eeq
with \,$ 0 \le c_i \le a_i $, $ i = 1, \dots , m$, and
\beq \label{eq.cugl}
  c_1 \le c_2 \le \cdots \le c_m, \quad {\rm{and}} \quad
a_1 - c_1 \le a_2 - c_2 \le  \cdots \le a_m -  c_m .
\eeq

It is known that marked subspaces can be characterized
in a basis free manner.

\begin{theorem} {\rm{(see \cite{AW3})}} \label{thm.vlski}
A subspace $W \in \Inv (V) $  is marked if and only if
\[
 f ^s W \cap f^{s +r}V = f^s (W \cap f^r V)
\]
for all $s \ge 0$, $ r \ge 0 $.
\end{theorem}
A different characterization of marked subspaces  can be found in
\cite{FPP}.

\section{Characteristic subspaces}    \label{sct.chrst}

\subsection{Hyperinvariant subspaces and projections} \label{sbsct.hypr}

Let
\,$  U = \bigl( U_{a_1}, \dots , U_{a_m} \bigr) \in \mcU  $\,
be given as in  \eqref{eq.prtt} such that
\,$  \e (  U_{ a_{\mu} } ) =  a_{\mu}$,
$ |   U_{ a_{\mu} } |  = d(  a_{\mu})$,
 $ \mu = 1, \dots, m$, and
let \eqref{eq.fvp},
i.e.
\,$
V = \langle U_{a_1}    \rangle  \oplus \cdots \oplus
  \langle  U_{a_m}   \rangle $,
 be  the corresponding decomposition of $V$.
In the following we are concerned with characteristic
subspaces $X$ which have the property that
$ x \in X $,  implies
\beq \label{eq.umdef}
 \pi _i x \in X \quad {\rm{for \: \: \: all}} \quad i = 1, \dots , m.
 \eeq
We have seen in Example~\ref{ex.fu22} that \eqref{eq.umdef}
 is not satisfied
for all  $ X \in {\rm{Chinv}}(V) $.

 \medskip

\begin{lemma} \label{la.essntl}

If  $X $ is a characteristic subspace  of $V$ then
\[
 X \cap \langle  U_{a_i}    \rangle  = f^{c_i}  \langle U_{a_i} \rangle ,
\:\:\: i = 1, \dots , m ,
\]
for some $ 0 \le c_i \le a_i $.
\end{lemma}

\begin{proof}
Let
$  U_{a_i} = ( v_1, \dots , v_{\ell} ) $.
Set $ X_{ i } = X \cap       \langle U_{a_i}    \rangle $.
Assume $ X_{ i } \ne 0 $,
and $\h(  X_i ) = c_i $.
Then
\,$      X _i \subseteq f^ {c_i}   \langle U_{a_i}    \rangle $.
Suppose $ y \in X_i $ and
 $ \h(y ) = c_i $.
Then
$ y = f^{c_i} w $ for some $ w \in V$ with  $ \h(w) = 0$.
We have
 $w = w_1 + \cdots + w_m $,
 $ w_j \in  \langle  U_{a_j}   \rangle $, $ j = 1, \dots , m$.
Then
\,$f ^{c_i}  w =  f ^{c_i}  w_1 +
 \cdots +   f ^{c_i}  w_m \in  X_i $
and \eqref{eq.fvp} 
 imply
     \,$f^c  w_j = 0 $ if $ j\ne i$.
Hence
\,$ y  = f^{c_i} w_i $, $ w_i \in  U_{a_i}   $, $\h(w_i) = 0$.
By  Lemma~\ref{la.pd13}  we can   replace some vector
  in  $ U_{a_i}  $ by $w_i$. Thus,
without loss of generality we can  assume
$ y =  f^{c_i} v_1 $.
Let $ \alpha_k \in  {\rm{Aut}}_f(V)$ be an automorphism which
maps $ v_1 $ to $v_k$, $k = 1, \dots ,\ell$.   Since
$ X $ is characteristic we obtain $  \alpha _k (y) = f^{c_i} v_k \in
X $. Therefore
\[
 f^{c_i}   \langle v_1, \dots , v_{\ell} \rangle =
 f^{c_i}    \langle U_{a_i}      \rangle  \subseteq X _i .
\]
Hence we have shown that $ X_i =  f^{c_i}   \langle U_{a_i} \rangle   $.
\end{proof}

\begin{lemma} \label{la.usa}
Let
 $ X $ be a characteristic   subspace of $V$.
The following statements are equivalent.
\begin{itemize}
\item[\rm{(i)}]  $X $ is hyperinvariant.
\item[\rm{(ii)}]
We have
 \beq \label{eq.drxsch}
X =  \big( X \cap \langle U_{a_1} \rangle  \big)
\, \oplus \, \cdots \, \oplus \,  \big(  X \cap \langle U_{a_m}
 \rangle \big).
\eeq
\item[\rm{(iii)}]
If
$ x \in X $,
 $ x = x_1 \, + \cdots + \, x_m $, $x_i \in  \langle U_{a_i} \rangle$,
$i = 1, \dots , m$,
then
\beq \label{eq.prpop}
  x_i
\in X \quad {\rm{for \: \: \: all}} \quad i = 1,
\dots , m.
\eeq
\end{itemize}
\end{lemma}

\begin{proof}
 (ii) $\LR$  (iii).
Because of
$  X \cap  \langle U_{a_i} \rangle \subseteq   \pi _i X  $
we have  $ \pi _i X \subseteq  X $ if and only if
\beq \label{eq.pewx}
 \pi _i X = X \cap  \langle U_{a_i} \rangle .
\eeq
It is obvious that
 (iii), as well as (ii), is satisfied
if and only if \eqref{eq.pewx} holds for all $i$, $ i = 1,
\dots , m$.
\\
(ii) $\Ra $ (i)
If   \eqref{eq.drxsch} holds then  Lemma~\ref{la.essntl} implies
\[
X =  f^{c_1}  \langle U_{a_1} \rangle \, \oplus \, \cdots \, \oplus \,
 f^{c_m}  \langle U_{a_m} \rangle .
\]
Hence $X$ is marked,
and it follows from Theorem~\ref{thm.chhym}
that $X$ is hyperinvariant.
\\
(i) $\Ra $ (ii)
Set $ \tilde{X} =
\oplus _{i = 1} ^m ( X \cap \langle U_{a_i}  \rangle )$.
Then $ \tilde{X} \subseteq X $.
Let $x \in X$. Then   $  \pi _{\mu} \in  {\rm{End}}_f(V) $
implies  $  \pi _{\mu}  x  = x_{\mu} \in ( X  \cap \langle U_{a_i}  \rangle )$,
$ \mu = 1,\dots , m$.  Hence $ X \subseteq  \tilde{X}  $.
\end{proof}

\bigskip

We extend Lemma~\ref{la.usa}
to the case where
If $X \in {\rm{Chinv}}(V) \setminus  {\rm{Hinv}}(V)$.

\begin{theorem} \label{thm.lgsth}
If $X $ is a characteristic subspace
then
\beq \label{eq.larg}
 \tilde{X} = \big( X \cap \langle U_{a_1}  \rangle
\big)
\oplus
\dots \oplus
\big(X \cap \langle U_{a_m}  \rangle\big)
\eeq
is the largest hyperinvariant subspace contained in $X$.
\end{theorem}

\begin{proof}
From  Lemma \ref{la.essntl}  we obtain
\beq
\label{eq.nnzvu}
  X \cap \langle U_{a_i} \rangle    =
f^{c_i} \langle U_{a_i} \rangle  , \;  0 \leq c_i \leq a_i,
\; i = 1, \dots, m .
\eeq
 Hence
\beq \label{eq.cvhrn}
 \tilde{X} =
f^{c_1} \langle U_{a_1}   \rangle
\oplus
\dots \oplus f^{c_m} \langle U_{a_m}  \rangle .
\eeq
We show that  $ \tilde{X}  $ is hyperinvariant.
By    Theorem~\ref{thm.qv}
we have to prove that
\,$i < j$\, implies
\beq \label{eq.scinq}
 c_i \leq c_j \quad {\rm{and}} \quad a_i -c_i \leq a_j - c_j.
\eeq
Suppose  $i <j$ and let  $v_i \in  U_{a_i} , v_j \in  U_{a_j} $.
By Lemma~\ref{la.mma8}
 there exists an  $ \alpha \in \Aut_f(V)$
such that  $\alpha v_j = v_i+v_j$.
Since $ X $ is characteristic and
\,$ f^{c_j}v_j \in \tilde{X}  \subseteq  X  $\,
we have
 $\alpha(f^{c_j}v_j)= f^{c_j}(v_i+v_j) \in X$.
Thus
\[ f^{c_j} v_j \in f^{c_j} \langle U_{a_j} \rangle  = X \cap
  \langle U_{a_j} \rangle \subseteq X
\]
 implies
$  f^{c_j}v_i \in X  $. Hence
 $  f^{c_j}v_i \in X  \cap \langle U_{a_i}  \rangle $.
Then \eqref{eq.nnzvu} yields  \,$ c_j \geq c_i$.

The second inequality in \eqref{eq.scinq}
can be proved as follows.
Because of
 \,$
\e( v_i ) = a_i = \e ( v_i +f^{a_j-a_i}v_j  )
$\,
 we can substitute
$ v_i $  in  $ U_{a_i} $  by  $ v_i +f^{a_j-a_i}v_j $.
Then  there exists an $ \alpha \in \Aut_f(V)$
with  $  \alpha v_i =  v_i +f^{a_j-a_i}v_j $.
Hence
\[  \alpha f^{c_i} v_i =  f^{c_i} v_i +f^{c_i+ a_j-a_i}v_j .
\]
Because of  \,$f^{c_i}v_i \in X$\, we have
\,$f^{c_i} v_i +f^{c_i+ a_j-a_i}v_j \in X$, and therefore
$f^{c_i+ a_j-a_i}v_j \in X$.
Hence  $f^{c_i+ a_j-a_i}v_j \in \langle U_{a_j} \rangle$ implies
 $f^{c_i+ a_j-a_i}v_j \in X \cap \langle U_{a_j} \rangle$.
Then \eqref{eq.nnzvu} yields
 $c_i + (a_j-a_i) \geq c_j $,
i.e.  $a_i -c_i \leq a_j - c_j$.

It remains to show that all  hyperinvariant subspaces contained
in $X$ are subsets of  $ \tilde{X} $.
Let $ W \in {\rm{Hinv}}(V) $.   
Then it follows from  Lemma~\ref{la.usa}  that
\[
   W =
   (W  \cap  \langle U_{a_1} \rangle  )
\oplus \cdots \oplus   (W  \cap  \langle U_{a_m} \rangle  )
=
f^{d_1} \langle U_{a_1}   \rangle
\oplus
\dots \oplus f^{d_m} \langle U_{a_m}  \rangle ,
\]
with suitable integers $0 \le d_i \le a_i$.
Suppose $ W  \subseteq X $. Then
\[
 f^{d_i} \langle U_{a_i}  \rangle =   W   \cap  \langle U_{a_i} \rangle
 \subseteq X  \cap  \langle U_{a_i} \rangle  =
f^{c_i} \langle U_{a_i}  \rangle .
\]
Therefore  \eqref{eq.cvhrn} implies  $ W  \subseteq \tilde{X} $.
\end{proof}

\subsection{Special cases}

In the following lemma we assume that $f$ is such that
\beq \label{eq.dakr}
 d(a) = k > 1 , \quad   {\mbox{and}}  \quad   d(r) = 0 \quad
  {\mbox{if}} \quad   r \ne a,
\eeq
or equivalently $ V[f] = f^{a-1}V  $ and $\dim  V[f] = k > 1 $.
In that case
there are $k$ blocks in the
Jordan form of $f$ and all Jordan blocks of $f$ have size~$a$.

\begin{lemma} \label{la.nzbw}
Assume \eqref{eq.dakr}.
Then there exist $ \beta, \gamma \in {\rm{Aut}}_f(V)$
such that
\,$
    \beta +  \gamma =    \iota $.
\end{lemma}

\begin{proof}
Let $N_a$ be the Jordan block \eqref{eq.npjo}.
Then the Jordan form of $f$ is
\[
  J = \diag(N_a, \dots , N_a) =  I_k \otimes N_{a} ,
\]
and it is no loss of generality to assume $ f = J$.
Define \,
$M   = (N_a \otimes I_{ k} )  - (N_a ^T \otimes I_{k}) $.
Then
\[
M = \bpm
0 &  -   I_{k} & 0  &  &  &
\\
 I_{ k} & 0 &  - I_{ k}  &  &   &
\\
 &   \ddots & \ddots  & \ddots   &  &
\\
 &   &  \ddots  & \ddots   & \ddots  &
\\
 &   &   &  \ddots & 0  & - I_{ k}
\\
 &   &   &   & I_{ k}   &  0
\epm   _{ak \times ak}
\]
and \,$ MJ = JM $.
Set
\[
 P_1 = \diag(1, 0, \dots , 0) _{a \times a} \otimes  I_{ k}
=
\bpm
 I_{k} &   &         &  &
\\
 &  0   &       &   &
\\
 &    &       & \ddots  &
\\
 &   &   &        &  0
\epm _{ak \times ak} ,
\]
\[
 P_c =  \diag(0, 1 \dots , 1) _{a \times a} \otimes  I_{ k}
=
\bpm
0 &   &         &  &
\\
 &  I_k   &       &   &
\\
 &    &       & \ddots  &
\\
 &   &   &        &  I_k
\epm _{ak \times ak} ,
\]
and
\,$ \beta  = M + P_1$, \: $ \gamma   = M + P_c$.
Then
$  \beta  ,  \gamma \in {\rm{Aut}}_f(V) $,
and
\,$   \beta  +  \gamma   =  \iota $.
\end{proof}

\begin{lemma} \label{la.schl}
Suppose $X $ is a characteristic subspace of $V$.
Let
$ x \in X $ and
\[
 x = x_1 + \cdots + x_m ,
\:
  x_i \in \langle U_{a_i}  \rangle,
\,
 i = 1, \dots ,m .
\]
If \,\,$ |  U_{a_s} |  > 1 $\, then \,$ x_s \in X $.
\end{lemma}

\begin{proof}
According to Lemma \ref{la.nzbw}
there exist $ \beta _s , \gamma_s \in
 {\rm{Aut}}_f( \langle U_s  \rangle) $
such that
\[
  \beta _s + \gamma _s
 = {\rm{id}}_{ \langle U_s  \rangle} .
\]
Let
 $ \psi :  V \to V $
and  $ \phi  :  V \to V $  be given by
\[
  \psi v =  \phi v = v \quad {\rm{for}} \quad v \in
  \langle U_1, \dots , U_{s-1} , U_{s+1} ,
\dots , U_m \rangle
\]
and
\[
 \psi v = \beta _s v , \: \phi v = \gamma _s v
 \quad {\rm{for}} \quad v \in U_s.
\]
Then $  \psi, \phi  \in  {\rm{Aut}}_f(V)$.
Therefore
\,$
( \psi + \phi ) x = ( \beta _s + \gamma _s) x_s = x_s \in X$.
\end{proof}

The following two theorems, which involve special
types of  Ulm sequences,  will cover the  hypothesis (i)
of Theorem~\ref{thm.ulmkap}.
It should be mentioned that the proofs  of
Theorem~\ref{thm.eintl} and Theorem~\ref{thm.nfn}
below employ marked subspaces and thus
are based  on Theorem~\ref{thm.chhym}.

\begin{theorem} \label{thm.eintl}
If the sequence \eqref{eq.ulmsq}
contains at most one  Ulm invariant
with  \mbox{$d(i) = 1$},
then each characteristic subspace of $V$ is hyperinvariant.
\end{theorem}

\begin{proof}
Suppose   $ |U_{a_q}| \ge   1$,  and  $ | U_{a_i} | > 1 $ if  $i \ne q$.
Let $ x \in X $ and
\,$ x = x_1 + \cdots + x_m $,
  $ x_i \in \langle U_{a_i}    \rangle$,
$ i = 1, \dots ,m $.
Then
Lemma \ref{la.schl} implies
$x_i \in X $ if $ i \ne q $.
Therefore, also $ x_q \in X $, and we obtain
\eqref{eq.prpop}.
 Then Lemma \ref{la.usa}
completes the proof.
\end{proof}

\bigskip

Bru,  Rodman and Schneider \cite[Thm.3.4, p.223]{Bru}
have shown  (see also \cite[3.2.4, p.\,28]{FPP})
that
each invariant subspace of $V$ is marked if and only if
the sizes of blocks in the Jordan form of $f$ differ at most
by one.
Then,  for some $q$ the space
$V$ is of the form   \,$ V = \langle U_{q} \rangle $\, or
 \,$ V = \langle U_{q} \rangle \oplus \langle U_{q + 1} \rangle$.
We only  need  the special case
where
 $ | U_{q}  | = |   U_{q + 1}  | = 1 $.
Then $f$ has Jordan form $ J = \diag(N_{q},
N_{q + 1}) $.
For the sake of completeness we
consider  that case  in  a lemma and include a proof.

\begin{lemma}  \label{la.lemma5}
Let
\beq \label{eq.stqq}
V = \langle u_1   \rangle  \oplus \langle u_2   \rangle
\quad {\rm{and}} \quad
\e(u_1) = q, \: \e(u_2) = q +1 .
\eeq
Then {\rm{(i)}}  each invariant subspace of $V$ is marked,
 {\rm{(ii)}} each characteristic subspace of  $V$  is hyperinvariant.
\end{lemma}

 \begin{proof}
(i)
 Let $W \in \Inv (V)  $, $ W \ne 0$.
Set $ \h(W) = \min\{\h(w); \, w \in W, \, w \ne 0\} $.
It is easy to see that it suffices to consider subspaces
$W$ with $  \h(W) = 0 $.
 Suppose $W $ is cyclic, $ W = \langle w \rangle $
and $  \h(w) = 0 $.  Then  $\e(w) = q $ or  $\e(w) = q +1 $.
In the first  case   we have
$ (w, u_2) \in \mcU $, in the second case we obtain
$ (u_1, w) \in \mcU $. Thus  $\langle w \rangle $
is marked.

Now suppose $W$ is not cyclic and $  \h(W) = 0 $.
Then
$W =  \langle  w_1  \rangle  \oplus  \langle  w_2  \rangle $,
 $w_1 \ne 0 $, $w_2 \ne 0 $,
and
\,$
\min\{ \h(w_1),  \h(w_2) \} = 0 $.
Suppose
$  \h(w_1) = 0 $. 
If $ \e(w_1) = q $ then    we have
$ (w_1, u_2 ) \in \mcU$, and we can assume
$ w_1 = u_1$,
such that
$W =   \langle  u_1  \rangle  \oplus  \langle  w_2  \rangle$.
If $ w_2 = z_1 + z_2 $, $z_i \in \langle u_i \rangle $,
$ i = 1,2$,
then
$ W =   \langle  u_1  \rangle  \oplus  \langle  z_2  \rangle$.
Let $\h( z_2 )= r $. Then  $ z_2 = f^r v_2 $, where  $v_2 \in
 \langle u_2 \rangle $, $\h(v_2) = 0 $.
Hence $ \e(v_2) = q + 1 $ and $ (u_1 , v_2 ) \in \mcU $.
Therefore  $W =  \langle  u_1  \rangle  \oplus  f^r
 \langle  v_2  \rangle$.
A similar argument  works in the  case   $ \e(w_1) = q +1  $.
(ii)  This follows from Theorem~\ref{thm.chhym}.
 \end{proof}

\bigskip

Part (ii) of the preceding lemma is a special case of the
following result.

\begin{theorem}  \label{thm.nfn}
Suppose   the Ulm sequence  \eqref{eq.ulmsq} contains
exactly two invariants $ d(i) $ and  $ d(j) $ equal to $1$,
and $i $ and $j $ are successive integers.
Then each characteristic subspace  $ X \subseteq V $ is
hyperinvariant.
\end{theorem}

\begin{proof}
We can assume
$  | U_{a_s} | = | U_{ a_{ s+1}  } | = 1 $,
$ a_s = q $, $  a_{ s+1} = q +1$,
and
$  | U_{a_{ \mu }} |  > 1 $ if $ a_{ \mu} \ne a_ s$ and
 $a_{ \mu}  \ne  a_{ s+ 1}   $.
Suppose  $ X  \subseteq V$ is characteristic.
Let $x \in X $ be
 decomposed  as
\beq \label{eq.brxx}
 x = x_1 + \cdots + x_{s -1 } +
  ( x_s + x _{s +1 } ) + x_{s+2}  + \cdots + x_m  ,
\eeq
$x_{\mu } \in  \langle U_{a_{ \mu }}  \rangle $.
Then Lemma~\ref{la.schl}
implies  $ x_ \mu \in X $ if $  \mu \ne s,  \mu \ne s+1 $.
Hence
$  x_s + x _{s +1 } \in X $.
Set
\,$ X_{\mu} = X \cap \langle U_{a_{ \mu }}  \rangle $.
Then
\[
X = X_1 \oplus \cdots \oplus  X_{s-1} \oplus
  (X \cap  \langle U_{a_s}, U_{a_{s+1}} \rangle )
\oplus
X_{s +2} \oplus \cdots \oplus  X_m .
\]
  Lemma \ref{la.essntl}  yields
\beq \label{eq.alauss}
X_{\mu}  = f^{c_{\mu}}
 \langle U_{a_{\mu}}  \rangle , \: \:
{\rm{if}}
 \: \:
  \mu \ne s,  \,  \mu  \ne s+1 .
\eeq
Let $ \hat{f} $ be the restriction of $f$ on
$  \langle U_{a_s} ,  U_{a_{s+1}} \rangle $.
We  show that the subspace        
 \,$  X_{s,s+1}  := X \cap
 \langle U_{a_s} ,  U_{a_{s+1}} \rangle $\,
 is characteristic in
$ \langle U_{a_s} ,  U_{a_{s+1}}  \rangle $
with respect to  $  \hat{f} $.
Let  $ \hat{\alpha} $ be an automorphism of
$ \langle U_{a_s} ,  U_{a_{s+1}}  \rangle $ that commutes
with  $  \hat{f} $,
and let
$ \imath_{\mu} $, $ \mu = 1, \dots, m$,  be    the identity
map on $\langle U_{a_{\mu}} \rangle$.
We extend $  \hat{\alpha} $ in a natural way to
an automorphism $ \alpha $ of $V$
such that
\[
\alpha = \imath_1 +  \dots + \imath_{s-1}  +  \hat{ \alpha } +
 \imath_{s+2}  +   \dots + \imath_m
\in \Aut_f(V).
\]
 If \,$x \in X_{s,s+1}$\,    then  
\, $ \hat{\alpha} x = \alpha x \in X$.
Thus, $\hat{\alpha} x \in X_{s,s+1}$, which implies that
$ X_{s,s+1} $ is characteristic in
$ \langle U_{a_s} , U_{a_{s}+1} \rangle $
with respect to  $ \hat{f} $.
The pair  $(  U_{a_s} , U_{a_{s}+1}  )$ is a generator tuple of
$ \langle  U_{a_s} , U_{a_{s}+1} \rangle $.
From  Lemma \ref{la.lemma5} we know that the characteristic
subspace $ X_{s,s+1}$ is hyperinvariant and therefore
marked in $  \langle U_s , U_{s+1} \rangle $.
Hence
\beq \label{eq.snso}
X \cap   \langle U_s , U_{s+1} \rangle =
  X_{s,s+1} = f^{c_s}  \langle \tilde{U}_s  \rangle \oplus
f^{c_{s+1}}  \langle \tilde{U}_{s+1} \rangle  ,
\eeq
with $ \langle U_s , U_{s+1} \rangle =
 \langle \tilde{U}_s ,   \tilde{U}_{s+1} \rangle $.
Combining \eqref{eq.alauss} and  \eqref{eq.snso}
shows that
$X$ is marked. Therefore,  by Theorem \ref{thm.chhym},
the subspace $X$ is hyperinvariant.
\end{proof}

\subsection{Characteristic but  not hyperinvariant}

Theorem~\ref{thm.mmnon} below is
 crucial for  a proof
of  the implication ``(ii) $\Ra$ (i)''
of The\-orem~\ref{thm.vnpsa}.
We first note  a technical lemma, which is adapted from
 \cite[p.\,63]{Kap}. It  clears  the way for Theorem~\ref{thm.mmnon}.
Define
 \[
 \langle \bar{U}_{[i,j]} \rangle  =
\langle  U_{a_i}, \dots , U_{a_j}  \rangle , \:\:
  1 \le  i \le j \le m  .
\]

\begin{lemma}  \label{la.prtt}
 Let
\,$
  U =   \bigl( U_{a_1},  \dots ,
 U_{a_m} \bigr) \in \mcU $,
and
\beq
\label{eq.rstm}
| U _{ a_ { \rho } } | = | U _{ a_{\tau } } | = 1,
\:\:
 a_{ \rho} + 1 < a_{ \tau } .
\eeq
Let $ U _{ a_ { \rho } } = (  u_{(\rho)} ) $,
$ U _{ a_{\tau } }  = (   u_{(\tau)}  ) $.
Define
\beq \label{eq.zdefn}
 z =  f^ { a_{\rho } - 1 }  u_{(\rho)}
+  f^ { a_{\tau }  - 2 }   u_{(\tau)}
\eeq
and
\beq \label{eq.ylamk}
 Y =
 \{ y  \in V ; \:\,  \e(y) = 2 ,
\, \,
   \h(y) =   a_ { \rho } -1 ,
\,  \,
\h( fy ) =   a_ { \tau } -1 \} .
\eeq
Then
$ \langle  Y \rangle$ is characteristic, and
\beq \label{eq.zwdmd}
 \langle  Y \rangle
  =  \langle z  \rangle
   \, \oplus \,
 \langle \bar{U}_{[\rho +1,  \tau - 1 ]} \rangle [ f ]
 \,\oplus \,
\langle \bar{U}_{[ \tau + 1 , m]} \rangle [ f^ 2 ] .
 \eeq
\end{lemma}

\medskip

\begin{proof}
The subspace  $\langle  Y \rangle $
is defined   via exponent and height. Hence it is
characteristic.
Set \[
  Q = \langle z  \rangle
   \, \oplus \,  \langle \bar{U}_{[\rho +1,  \tau - 1 ]} \rangle [ f ]
 \, \oplus \,
\langle \bar{U}_{[ \tau + 1 , m]} \rangle [ f^ 2 ] .
\]
We  first show  that   $ Y \subseteq Q $.
Let  $ y \in Y $,
\[
y \,  = x_1 + \cdots + x_m
 \,\;\:
x_i \in  \langle   U_{ a_i  }  \rangle , \, i = 1, \dots , m.
\]
Put $ x_{[i, j ]} = x_i +  \cdots + x_j $, $1 \le i \le j \le m$.
From
\,$\h(y) = a_\rho - 1 $\,
follows
\[
 y \in \oplus _{i=1}^m \,
  f^{ a_\rho - 1 }  \langle   U_{ a_i  } \rangle
=
  f^{ a_\rho - 1 }  \langle    u_{(\rho)}    \rangle
 \oplus
 f^{ a_\rho - 1 }  \langle  \langle \bar{U}_{[ \rho  + 1 , m]} \rangle ,
\]
and
\,$\h(fy) = a_\tau - 1 $\, implies
\[
fy \in   f^{ a_\tau - 1 }  \oplus _{i=\rho }^m
   \langle   U_{ a_i  } \rangle
=
 f^{ a_\tau - 1 }
  \oplus _{i= \tau  }^m   \langle   U_{ a_i  } \rangle .
\]
Therefore \,$ f x_i = 0 $, $ i = \rho , \dots, \tau -1 $,
and
\beq \label{eq.enmplf}
y \in
f^{ a_\rho - 1 }  \langle  u _{ (\rho ) } \rangle
\oplus
 \langle \bar{U}_{[\rho +1 , \tau -1 ]} \rangle [f] \oplus
 f^{ a_\tau - 2 } \langle \bar{U}_{[\tau , m ]} \rangle .
\eeq
From
 \,$ \e(y) = 2 $\,  we obtain
\[
y \in
 f^{ a_\rho - 1}   \langle  u _{ (\rho ) } \rangle
\oplus  \langle \bar{U}_{[\rho +1 , \tau -1 ]} \rangle [f]
\oplus
 \langle \bar{U}_{[\tau , m]} \rangle  [f^2]  .
\]
We have
\[
 \langle \bar{U}_{[\rho +1 , \tau -1 ]} \rangle [f]
=
\oplus _{i = \rho + 1 } ^{\tau -1 }  \,  f^{ a_i- 1 }
\langle U_{ a_{i } }  \rangle  \subseteq  f^{ a_{\rho + 1 } -1}
\oplus _{i = \rho + 1 } ^{\tau -1 }
\langle U_{ a_{i } } \rangle   \subseteq
  f^{ a_{\rho }}  V .
\]
The assumption
$a_\tau >  a_\rho  + 1  $ implies
\[
 \langle \bar{U}_{[\tau , m]} \rangle  [f^2]
=
\oplus _{i = \tau } ^ m f^{ a_i- 2 } \langle U_{ a_{i } } \rangle
\subseteq
 \oplus _{i = \tau } ^ m f^{ a_\tau - 2 } \langle U_{ a_{i } } \rangle
\subseteq   f^{ a_{\rho }}  V
\]
and
\beq \label{eq.serp}
 \langle \bar{U}_{[\tau +1  , m]} \rangle  [f^2]
\subseteq
 \oplus _{i = \tau +1 } ^ m
 f^{ a_{\tau +1}  - 2 } \langle U_{ a_{i } } \rangle
\subseteq  \oplus _{i = \tau +1 } ^ m
   f^{ a_{\tau}  - 1 } \langle U_{ a_{i } } \rangle .
\eeq
Hence  \eqref{eq.enmplf} and
$ h(y) =  a_\rho -1  $ yield  $ x_\rho \ne 0 $,
i.e. $ x_\rho =   f^{ a_\rho - 1 }  u _{ (\rho ) } $.
Then
\begin{multline*}
y =  \big( f^{ a_\rho - 1 }  u _{ (\rho ) }  +
(x_{[\rho + 1 , \tau -1]} \big) +
\big( x_\tau  +  x_{[\tau +1 , m]} \big)
=   x_{[\rho , \tau -1]}  +  x_{[\tau  , m]}  ,
\\
x_{[\rho , \tau -1]} \in
 \langle \bar{U}_{[\rho ,\tau -1]} \rangle  [f] ,
\: \:
 x_{[\tau + 1 , m]} \in  \bar{U}_{[\tau +1 ,m]} \rangle  [f^2] .
\end{multline*}
From $ \e(y) = 2 $ and $ f x_{[\rho , \tau -1]} = 0 $ follows
$  x_{[\tau , m]} \ne 0 $,  $\e (   x_{[\tau  , m]} ) = 2 $,
and  $ f y = f (  x_{[\tau  , m]} ) $.
Therefore
$  x_\tau  \ne 0 $. Otherwise
$ x_{[\tau +1 , m]} \ne 0 $, and  then
\eqref{eq.serp}
would imply  $ \h(fy ) = \h (f   x_{[\tau +1 , m]})
\ge  a_{\tau} $.
Hence
\[
 x_\tau =  f^{ a_\tau - 2 }  u_{(\tau)}  + \gamma
  f^{ a_\tau - 1 }  u_{(\tau)}, \: \gamma \in \{0,1\} .
\]
Putting the pieces together we obtain
\beq \label{eq.mrchg}
y =  \big(
f^{ a_\rho  - 1 }  u_{(\rho)}   + f^{ a_\tau - 2 }  u_{(\tau)}  + \gamma
  f^{ a_\tau - 1 }  u_{(\tau)} \big) +
x_{[\rho + 1, \tau -1 ]} + x_{[\tau +1, m ]},
\eeq
and
\beq \label{eq.alab}
x_{[\rho + 1, \tau -1 ]} \in
  \langle  U_{[  \rho + 1 ] , \tau - 1] }  \rangle [f] ,
\:\:
x_{[\tau +1, m ]} \in \langle
 U_{[\tau +1, m ]}   \rangle [f^2],
\eeq
and
\[
f^{ a_\rho  - 1 }  u_{(\rho)}   + f^{ a_\tau - 2 }  u_{(\tau)}  + \gamma
  f^{ a_\tau - 1 }  u_{(\tau)}
= z + \gamma fz  \in \langle z \rangle .
\]
Hence
 $y \in Q $, and  $  \langle  Y\rangle  \subseteq Q $.
The space
$Q$ is generated by vectors  of the form
 \beq \label{eq.typv}
 z +x_{[\rho + 1, \tau -1 ]} + x_{[\tau +1, m ]},
\;\:
 z + fz +
x_{[\rho + 1, \tau -1 ]} + x_{[\tau +1, m ]},
\eeq
where $ x_{[\rho + 1,\tau -1 ]} $ and $ x_{[\tau +1, m ]}$
satisfy \eqref{eq.alab}.
It is easy to see that the vectors \eqref{eq.typv}
lie in $Y $. Hence
$ Q \subseteq   \langle  Y\rangle $.
\end{proof}

In general, if \eqref{eq.rstm} holds then
there exists more  than one  characteristic subspace
of  $V$  that is
not hyperinvariant.
In a subsequent paper  \cite{AW6} we construct a larger class of
  charcateristic non-hyperinvariant subspaces of $V$,
which includes $ \langle Y \rangle $
  as a special case.

\medskip

\noindent {\bf{Example \ref{ex.fu22} continued.}}
If  \,$| K |   = 2 $,  
\[
 V = K^4 =  \langle e_1  \rangle \oplus   \langle e_2 \rangle ,
\:\:
 \e (  e_1 ) = 1 ,  \,  \e (  e_2 ) = 3 ,
\]
then \eqref{eq.rstm}  holds with
 $(a_1, a_2) = ( 1, 3) $. Let  $X$ be the subspace
as in  \eqref{eq.ouszb}.
Set $ z = e_1 + f e_2 = e_1 + e_3 $ and
$Y =  \{ y ; \;  \e(y) = 2 , \, \h(y) = 0 , \, \h(fy) = 2   \} $.
Then $ X =  \langle z \rangle =   \langle Y  \rangle $.

If $ | K | > 2 $ then
\,$Y = \{c_1 e_1 + d_3 e_3 + d_4 e_4 ; \; c_1 \ne 0, \,
d_3 \ne 0 \}$
and
\[  \langle Y  \rangle = \mbox{span} \{e_1, e_3, e_4 \}  =
\Ker f^2 . \]
Then $  \langle Y  \rangle $ is hyperinvariant, and
 $  \langle z  \rangle \subsetneq
  \langle Y  \rangle $.

\medskip

\begin{theorem}  \label{thm.mmnon}
Suppose  $ K = GF(2)$. Assume that \eqref{eq.rstm} holds.
Let
\beq \label{eq.emalli}
 Y =
\{y;
\e(y) = 2, \:
\h( y) = a_\rho - 1,
 \;\:
\h(f y) = a_\tau - 1 \}.
\eeq
Then  the subspace $\langle Y \rangle $ is
characteristic and not hyperinvariant.
\end{theorem}

\begin{proof}
Let \,$\pi _\rho  : V \to V $\, be the projection
on  \,$ \langle  u_{(\rho)} \rangle $\, along
the complement
\,$  \langle U_{a_1},  \dots,  U_{a_{\rho - 1}},
U_{a_{\rho + 1}}, \dots, U_{a_m}   \rangle$.
Then $ \pi  _\rho   $ commutes with $f$.
If $ z $ is given by \eqref{eq.zdefn} then
$ z \in  \langle Y \rangle $ and
$  \pi _{\rho}  z = f^{a_{\rho} - 1}  u_{(\rho)} $.
Note that
\,$  f^{ a_{\rho} - 1}  u_{(\rho)} \notin
 \langle z \rangle $.
Therefore \eqref{eq.zwdmd}  implies
$  \pi _{\rho}  z \notin \langle Y \rangle $.
Hence  $\langle Y \rangle $ is not hyperinvariant.
\end{proof}

\section{Proof of Shoda's theorem}

We reformulate Theorem~\ref{thm.vnpsa}
in terms of Ulm invariants and thus  make the connection
with Theorem 26 of Kaplanski \cite[p.\,63]{Kap}.

\begin{theorem} {\rm{(\cite{Sh}, \cite{Kap})   }}
 \label{thm.ulmkap}
Let $ K = GF(2)$ and $ f : V \to V  $ be nilpotent.
Then the following statements are  equivalent.
\\
{\rm{(i)}}  At most two Ulm invariants of $ (V,f) $
are equal to  $1$, and if there are exactly two,
then they correspond to successive integers.
\\
 {\rm{(ii)}} The characteristic subspaces of $V$ are
hyperinvariant.
\end{theorem}

\begin{proof}
Suppose
condition (i)
 is satisfied. If there  exists  at most one Ulm invariant  equal to  $1$
then we can apply Theorem~\ref{thm.eintl}.
If there are exactly two  Ulm invariants
$ d(r) $ and $ d(s)  $
 equal to  $1$  and $ s = r + 1 $,
then we can apply  The\-orem~\ref{thm.nfn}.
If condition (i)
 is  not satisfied
then we have \eqref{eq.rstm} for some $\rho, \tau$.
Thus Theorem~\ref{thm.mmnon} completes the proof.
\end{proof}

\section*{Acknowledgment}
We are indebted to referees for valuable comments and
suggestions.


\begin{thebibliography}{99}


\bibitem{AW3} P.\ Astuti and H.\ K.\ Wimmer,
 %
Regular submodules of torsion modules over
a discrete valuation domain,
Czechoslovak.~Math.~J. 56 (2006), 349--357.



\bibitem{AW4} P.\ Astuti and H.\ K.\ Wimmer,
Hyperinvariant, characteristic and marked   subspaces,
 Oper. Matrices. 3  (2009), 261--270.

\bibitem{AW6} P.\ Astuti and H.\ K.\ Wimmer,
A class of characteristic invariant subspaces which are
 not hyperinvariant,
in preparation.


\bibitem{Bru}
R.\ Bru, L.\ Rodman, and H.\ Schneider,
Extensions of Jordan bases for in\-variant
 subspaces of a matrix, Linear Algebra  Appl. 150(1991),
 209--226.


\bibitem{CFP}
 A.\ Compta, J.\ Ferrer, and M.\ Pe\~{n}a,
Dimensions of the orbit of marked subspaces,
 Linear Algebra  Appl. 379(2004), 239--248.



\bibitem{FPP}
J.\ Ferrer,  F.\ Puerta, and X.\  Puerta,
 Geometric characterization and classification of marked subspaces.
Linear Algebra Appl.\ 235 (1996), 15--34.


\bibitem{GLR} I.\ Gohberg, P.\ Lancaster, and L.\ Rodman,
 Invariant Subspaces of Matrices with Applications,
Wiley, New York, 1986.




\bibitem{FuI} L.\ Fuchs, Infinite Abelian Groups, Vol. I.,
Academic Press, New York, 1973.


\bibitem{Kap0}
I.\ Kaplansky,
Some results on abelian groups,
Proc. Natl. Acad. Sci. USA,  38 (1952),
538-–540.

\bibitem{Kap}
I.\ Kaplansky,  Infinite Abelian Groups, University of Michigan
Press, Ann Arbor, 1954.


\bibitem{Lo}
E.\ Longstaff, Picturing the lattice of invariant subspaces
of a nilpotent complex matrix, Linear Algebra Appl.\ 56 (1984), 161--168.


\bibitem{Sh}
K.\ Shoda, \"Uber die characteristischen Untergruppen einer
endlichen Abelschen Gruppe,
 Math. Zeit. 31 (1930), 611--624.



\end{thebibliography}
\end{document}